\documentclass{amsart}
\usepackage[T1]{fontenc}
\usepackage{geometry}
\usepackage{hyperref}
\usepackage[    doi=false,
                url=false,
                isbn=false,
                style=alphabetic,
                maxnames=5
                ]{biblatex}
\usepackage{todonotes}
\usepackage{enumitem}
\usepackage{setspace}
\usepackage{tabu}
\usepackage{lmodern}
\usepackage{tikz}
\usepackage{amsmath,amssymb,amsthm}
\usepackage{MnSymbol}
\usepackage{color}
\usepackage{mathtools}
\usepackage{thmtools}
\usepackage{tikz-cd}
\usepackage[capitalise]{cleveref}
\usetikzlibrary{arrows,backgrounds,calc,positioning}

\newenvironment{thenum}{\begin{enumerate}[label= (\roman*),topsep=2pt,partopsep=0mm,parsep=0mm,itemsep=1pt]}{\end{enumerate}}

\crefname{enumi}{}{}
\Crefname{enumi}{}{}
\crefname{equation}{equation}{equations}
\Crefname{equation}{equation}{equations}
\declaretheorem{theorem}

\declaretheorem[sibling=theorem]{proposition}
\declaretheorem[sibling=theorem]{corollary}
\declaretheorem[style=definition]{definition}

\declaretheorem[numberwithin=theorem]{claim}
\declaretheorem{question}

\setlist{noitemsep}

\stdpunctuation%
\newcommand{\p}[1]{{\tilde{#1}}}
\newcommand{\B}{\mathcal{B}}

\newcommand{\A}{\mathcal{A}}

\newcommand{\ol}[1]{\overline{#1}}
\newcommand{\LR}{\Leftrightarrow}
\newcommand{\ra}{\rightarrow}

\newcommand{\eidomain}[2]{\mc{D}om_{\mc{#1}}^\mc{#2}}
\newcommand{\eidomstar}[2]{{\mc{D}om^*}_{\mc{#1}}^{\mc{#2}}}

\newcommand{\restrict}[1][]{{\upharpoonright}_{#1}}
\renewcommand{\phi}{\varphi}
\newcommand{\mc}[1]{\mathcal{#1}}
\newcommand{\mf}[1]{\mathfrak{#1}}

\newcommand{\Sicom}[1]{\Sigma^{\mathrm{c}}_{#1}}

\newcommand{\Decom}[1]{\Delta^{\mathrm{c}}_{#1}}

\DeclareMathOperator{\rng}{range}

\newcommand{\presentation}[1]{{\tilde{\mathcal{#1}}}}
\usetikzlibrary{decorations.markings}
\tikzset{negated/.style={
        decoration={markings,
            mark= at position 0.5 with {
                \node[transform shape] (tempnode) {$\backslash$};
            }
        },
        postaction={decorate}
    }
}
\title{On functors enumerating structures}
\subjclass[2010]{03C57}

\author{Dino Rossegger}
\address{Institute of Discrete Mathematics and Geometry, Vienna University of Technology, Austria}
\email{dino.rossegger@tuwien.ac.at}
\thanks{The author was supported by the Austrian Science Fund FWF through project P 27527.}
\urladdr{dmg.tuwien.ac.at/rossegger}

\begin{document}
 \begin{abstract}
   We study a new notion of reduction between structures called enumerable functors related to the recently investigated notion of computable functors. Our main result shows that enumerable functors and effective interpretability with the equivalence relation computable are equivalent. We also obtain results on the relation between enumerable and computable functors.
 \end{abstract}
\maketitle
\section{Introduction}\label{sec:intro}
In computable structure theory we study the algorithmic complexity of mathematical structures. One goal in this field is to compare structures, or classes of structures with respect to their computability theoretic properties. This is usually achieved by using reductions. Several different notions of reduction between structures are known, most notably Muchnik reducibility, Medvedev reducibility, computable functors, $\Sigma$-definability, and effective interpretability. The first three notions are computational, while the other two are syntactic, based on the model theoretic notion of interpretability. The study of computable functors was recently intiated by R.\ Miller, Poonen, Schoutens, and Shlapentokh~\cite{miller_computable_2015}. They are a strengthening of Medvedev reducibility. Harrison-Trainor, Melnikov, R.\ Miller, and Montalb\'an~\cite{harrison-trainor_computable_2017} showed that computable functors are equivalent to effective interpretability first studied by Montalb\'an~\cite{montalban_rice_2012}. In~\cite{harrison-trainor_borel_2016}, Harrison-Trainor, R.\ Miller, and Montalb\'an proved a similar result for Baire measurable functors and infinitary interpretability.
$\Sigma$-definability was introduced by Ershov~\cite{ershov_definability_????} and has since been heavily studied by Russian researchers~\cite{kalimullin_relations_2009,puzarenko_certain_2009,morozov_-definability_2008,stukachev_degrees_2007,stukachev_degrees_2008,stukachev_effective_2013}. Effective interpretability is equivalent to $\Sigma$-definability without parameters~\cite{montalban_rice_2012}.

Between classes of structures the most notable notions are computable embeddings, Turing computable embeddings, uniform transformations, HKSS interpretations and reduction by effective bi-interpretability. Turing computable embeddings~\cite{knight_turing_2007} are an analogue of Medvedev reducibility for classes of structures, while computable embeddings~\cite{calvert_comparing_2004} use enumeration reducibility, a well studied notion of reducibility in computability theory. Uniform transformations are based on computable functors and reduction by effective bi-interpretability~\cite{montalban_computability_2014} on effective interpretations. It was shown in~\cite{harrison-trainor_computable_2017} that these two notions are equivalent. Effective bi-interpretability is closely related to HKSS interpretations~\cite{hirschfeldt_degree_2002}. Hirschfeldt, Khoussainov, Shore, and Slinko~\cite{hirschfeldt_degree_2002} gave interpretations of graphs in several classes of structures. It turns out that with minor modifications of these interpretations one can obtain effective interpretations~\cite{montalban_computability_2014,rossegger_computable_2015}. As computable functors are a strengthening of Medvedev reducibility, uniform transformations are a strengthening of Turing computable embeddings.

In this paper we study \emph{enumerable functors}. Enumerable functors are a strengthening of computable embeddings. We prove that enumerable functors are at least as strong as computable functors and show that they are equivalent to a restricted version of effective interpretability if we focus on structures with universe $\omega$.  We obtain similar results for the related notions on classes of structures. The question if our notions are strictly stronger is still open.
\subsection{Notation}
Our computability theoretic notation is standard, see~\cite{soare_turing_2016} for background. Structures are in computable relational languages and have countable universe; we use calligraphic letters $\A,\B,\dots$ to denote structures and the matching upper case letters $A,B,\dots$ for their universes. If the context requires we write $R_i^\A$ to denote the interpretation of the $i^{\text{th}}$ relation symbol in the language $L$ in $\A$. We identify structures by their \emph{atomic diagram}. The atomic diagram of a structure $\A$ is the set of atomic sentences and negations of atomic sentences true of $\A$ under a fixed G\"odel numbering. If we want to emphasize that we are talking about the atomic diagram of $\A$ we write $D(\A)$.
A structure $\A$ is computable if its atomic diagram is computable. We usually write $\p\A,\hat\A$ for isomorphic copies of $\A$.

We denote categories by fraktal letters $\mf{C},\mf{D},\dots$. We write $\A\in \mf{C}$ to say that $\A$ is an object of $\mf C$ and $f\in \mf{C}$ means that $f$ is an arrow in $\mf{C}$. We introduce all further category theoretic definitions when needed.
\subsection{Enumerable functors}\label{sec:enumfunc} Recall the notion of a functor between categories. In our setting the categories are classes of countable structures, i.e., collections of structures closed under isomorphism with isomorphisms as arrows.
\begin{definition}
	A \textit{functor from $\mf{C}$ to $\mf{D}$} is a map $F$ that assigns to each structure $\A\in \mf{C}$ a structure $F(\A) \in \mf{D}$, and assigns to each arrow $f:\A\ra\B\ \in \mf{C}$ a morphism $F(f):F(\A)\ra F(\B)\in \mf{D}$ so that the following two properties hold.
	\begin{thenum}
		\item $F(id_{\A})=id_{F(\A)}$ for every $\A \in \mf{C}$, and
		\item $F(f\circ g) = F(f) \circ F(g)$ for all morphisms $f,g\in \mf{C}$.
	\end{thenum}
\end{definition}
We abuse notation and write $F:\A \ra\B$ for a functor $F$ between the isomorphism classes of $\A$ and $\B$. The \emph{isomorphism class} of $\A$ denoted by $Iso(\A)$ has as objects
\[ \left\{ \p\A \mid \p A=\omega \land \p\A\cong \A\right\}\]
and as arrows all the isomorphisms between copies of $\A$.
Enumeration reducibility is a well studied notion in classic computability theory that has also been studied in the context of computable structure theory, see~\cite{soskova_enumeration_2017} for a survey.
For $A,B\subseteq \omega$, $B$ is \textit{enumeration reducible} to $A$ if there is an enumeration operator, i.e., a c.e.\ set $\Psi$ of pairs $(\alpha, b)$ where $\alpha$ is a finite subset of $\omega$ and $b\in \omega$, such that
\[ B=\left\{ b \mid (\exists \alpha \subseteq A) (\alpha, b)\in \Psi\right\}.\]
We may write $B$ as $\Psi^A$ because $B$ is unique given $\Psi$ and $A$. Using an enumeration operator and a Turing operator we now define enumerable functors.
\begin{definition}\label{def:enumfunc}
	An \emph{enumerable functor} is a functor $F:\mf{C}  \ra \mf{D}$ together with an enumeration operator $\Psi$ and a Turing operator $\Phi_{*}$ such that
	\begin{thenum}
		\item\label{eq:enumfunc1stdef} for every $\A\in \mf{C},\ \Psi^{\A}=F(\A)$,
		\item for every morphism $f:\A\ra \B\in \mf C$, $\Phi_*^{\A\oplus f\oplus \B}=F(f)$.
	\end{thenum}
	As for computable functors we often identify enumerable functors with their pair $(\Psi, \Phi_{*})$ of operators.
\end{definition}
This effective version of functors is inspired by computable embeddings, investigated in~\cite{calvert_comparing_2004}. There, a computable embedding from a class $\mf{C}$ to a class $\mf{D}$ is an enumeration operator $\Psi$ as defined in \cref{eq:enumfunc1stdef} of \cref{def:enumfunc} and the property that $\A \cong \B$ if and only if $\Psi^{\A} \cong \Psi^{\B}$.
Our definition is stronger than this, since we additionally require isomorphisms $F(f):F(\A)\ra F(\B)$ to be uniformly computable from $\A\oplus f\oplus \B$.

The authors of~\cite{calvert_comparing_2004} showed that substructures are preserved by computable embeddings. The same observation can be made for enumerable functors. The proof is exactly the same, for sake of completeness we state it here.
\begin{proposition}\label{prop:substruct}[Substructure preservation]
	Let $F:\mf{C} \ra  \mf{D}$ be an enumerable functor witnessed by $(\Psi,\Phi_{*})$. If $\A_1,\A\in \mf{C}$ and $\A_1\subseteq\A$, then $F(\A_1)\subseteq F(\A)$.
\end{proposition}
\begin{proof}
	Assume $\A_1\subseteq\A$. If $\phi \in D(F(\A_1))$, then there is a finite set of formulas $\alpha \subseteq D(\A_1)$ such that $(\alpha,\phi)\in \Psi$ and since $D(\A_1)\subseteq D(\A)$, $\phi \in D(F(\A))$.
\end{proof}

As part of this article is concerned with the relationship between enumerable functors and computable functors we recall the notion of a computable functor first investigated in~\cite{miller_computable_2015}.
\begin{definition}\label{def:compfunc}
	A \emph{computable functor} is a functor $F:\mf{C} \ra \mf{D}$ together with two Turing operators $\Phi$ and $\Phi_{*}$ such that
	\begin{thenum}
		\item for every $\A\in \mf{C},\ \Phi^{\A}=F(\A)$,
		\item for every morphism $f:\A\ra \B\ \in \mf C$, $\Phi_*^{\A\oplus f\oplus \B}=F(f)$.
	\end{thenum}
	We often identify a computable functor with its pair $(\Phi, \Phi_{*})$ of Turing operators witnessing its computability.
\end{definition}
The following notions originated in~\cite{harrison-trainor_computable_2017}.
\begin{definition}\label{def:effeciso}
	A functor $F: \mf{C} \ra \mf{D}$ is \textit{effectively (naturally) isomorphic} to a functor $G: \mf{C} \ra \mf{D}$ if there is a Turing functional $\Lambda$ such that for every $\A\in \mf{C}$, $\Lambda^{\A}$ is an isomorphism from $F(\A)$ to $G(\A)$ and the following diagram commutes for every $\A,\B \in \mf{C}$ and every morphism $h: \A\ra \B$:\\
	\begin{center}
    \vspace{-0.3cm}
		\begin{tikzpicture}
			\tikzset{
				>=stealth,
				auto,
				node distance=1.5cm,
				vertex/.style={circle,fill,align=center, minimum size=5pt, inner sep=0pt},
				every loop/.style={looseness=10}
			}
			\node (fb) at (0,0){$F(\A)$};
			\node (gb) [right = of fb] {$G(\A)$};
			\node (fb2) [below = of fb] {$F(\B)$};
			\node (gb2) [below = of gb] {$G(\B)$};

			\path[]
			(fb) [->] edge node {$\Lambda^{\A}$} (gb)
			(fb2) [->] edge node {$\Lambda^{\B}$} (gb2)
			(fb) [->] edge node [left] {$F(h)$} (fb2)
			(gb) [->] edge node {$G(h)$} (gb2);
		\end{tikzpicture}
	\end{center}
\end{definition}
Note that in the above definition it does not matter whether $F$ and $G$ are both computable functors or enumerable functors. Hence, it is legal to say that an enumerable functor is effectively isomorphic to an computable functor. Intuitively, two functors are effectively naturally isomorphic if they are equivalent up to computable isomorphism. Using this idea one can generalize the idea of an inverse.

Let $F:\mf{C}\ra\mf{D}$ and $G:\mf{D}\ra\mf{C}$ be functors such that $G\circ F$ and $F\circ G$ are effectively isomorphic to the identity functors $id_\mf{C}$ and $id_\mf{D}$ respectively. Let $\Lambda_\mf{C}$ be the Turing functional witnessing the effective isomorphism between $G\circ F$ and the identity functor $id_\mf{C}$, i.e., for any $\A \in \mf{C}$, $\Lambda_\mf{C}^{\A}: \A \ra G(F(\A))$. Define $\Lambda_\mf{D}$ similarly, i.e., $\Lambda_\mf{D}^{\B}: \B \ra F(G(\B))$ for any $\B \in \mf{D}$.
Then there are maps $\Lambda_\mf{D}^{F(\A)}: F(\A) \ra F(G(F(\A)))$ and $F(\Lambda_\mf{C}^{\A}): F(\A) \ra F(G(F(\A)))$. If these two maps, and the similarly defined maps for $\mf{D}$, agree for every $\A\in \mf{C}$ and $\B \in \mf{D}$, then we say that $F$ and $G$ are \emph{pseudo-inverses}.
\begin{definition}\label{def:enumbitrans}
 Two structures $\A$ and $\B$ are \emph{enumerably bi-transformable} if there exist enumerable functors $F: \A \ra \B$ and $G:\B \ra \A$ which are pseudo-inverses.
\end{definition}
Harrison-Trainor, Melnikov, R.\ Miller, and Montalb\'an~\cite{harrison-trainor_computable_2017} defined \emph{computably bi-transformable} which is analogous to our definition except that it uses computable functors.
\subsection{Reduction between classes}
\begin{definition}
			A class $\mf{C}$ is \textit{uniformly enumerably transformally reducible}, or \emph{u.e.t.~reducible}, to a class $\mf{D}$ if there exists a subclass $\mf{D}'\subseteq \mf{D}$, enumerable functors $F:\mf{C}\ra \mf{D}'$ and $G:\mf{D}'\ra \mf{C}$, and $F,\ G$ are pseudo-inverses. We say that a class is \textit{complete for u.e.t.~reducibility} if for every computable language $L$, the class of $L$-structures u.e.t.~reduces to it.
\end{definition}
The authors of~\cite{harrison-trainor_computable_2017} gave an analogous definition using computable functors called reduction by uniform transformation. In this paper we call it \emph{reduction by uniform computable transformations}, or short \emph{reduction by u.c.t.}\ to distinguish it from reduction by uniform enumerable transformation. It is obtained by swapping the enumerable functors in the definition by computable functors.
\section{On the relation between enumerable and computable functors}
Kalimullin and Greenberg independently showed that if a class is computably embeddable in another class, then it is also Turing computably embeddable, see~\cite[Proposition 1.4]{knight_turing_2007}. Using a similar procedure to the one given in their proof one can construct a computable functor from an enumerable functor.
\begin{theorem}\label{th:enumimplcompfunc}
	Let $F:\mf{C} \ra \mf{D}$ be an enumerable functor, then there is a computable functor $G$ effectively isomorphic to $F$.
\end{theorem}

\begin{proof}
	Let $F$ be witnessed by $(\Psi,\Phi_{*})$. Given some $\A\in \mf{C}$, let $\B=F(\A)=\Psi^{\A}$. We first show that there is a Turing functional $\Phi'$ transforming every such $\A$ into a structure $\p\B$ isomorphic to $\B$, i.e., $\Phi'^{\A}=\p\B$.

	Let $\langle \cdot,\cdot \rangle:\omega \times\omega \ra \omega$ be the standard computable pairing function. The universe of $\p\B$ is
	\[\p B=\{ \langle b,s\rangle \mid b\in B \land s=\mu x [b=b\in \Psi_x^{\A}].\]
	Here $\Psi_x$ is the approximation of $\Psi$ at stage $x$ of the enumeration. $\p B$ is computable relative to $\A$ as computing membership can be done by enumerating $\Psi$ until stage $s$ and checking if $s$ is the first step such that $b=b\in \Psi_s^{\A}$.
	For each $R_i$ with arity $r_i$ in the language of $\mf{D}$ define $R_i^{\p\B}$ as
	\[
		\begin{array}{lcr}
			(\langle x_1,s_1 \rangle, \dots,\langle x_{r_i},s_{r_i}\rangle) \in R_i^{\p\B}     & \LR & R_i(x_1,\dots,x_{r_i}) \in \Psi^{\A},\\
			(\langle x_1,s_{1} \rangle, \dots,\langle x_{r_i},s_{r_i}\rangle) \not\in R_i^{\p\B} & \LR & \neg R_i(x_1,\dots,x_{r_i}) \in \Psi^{\A}.
		\end{array}
	\]
	Since for all relations $R_i$ and all tuples $(\langle x_1,s_{1} \rangle, \dots,\langle x_{r_i},s_{r_i}\rangle)\in \p B^{r_i}$, either $R_i(x_1,\dots,x_{r_i}) \in \Psi^{\A}$ or $\neg R_i(x_1,\dots,x_{r_i}) \in \Psi^{\A}$ and $\p B$ is computable from $\A$, $\p\B\leq_T \A$. Furthermore the computation of $\p\B$ from $\A$ is uniform, hence there is a Turing functional $\Phi'$ that given $\A\in \mf{C}$ as oracle computes $\p\B$. Set $G(\A)=\p \B$, then $\Phi'$ is the first partial witness of computability of $G$.

	For $\A\in \mf{C}$ let $\theta^{\A}: F(\A)\ra G(\A)$ be defined by $x \ra \langle x,s\rangle$. $\theta^{\A}$ is uniformly computable from $\A$ and is an isomorphism between $F(\A)$ and $G(\A)$ by construction of $G(\A)$.
	For the second partial witness consider $f: \A\ra \p\A\in \mf{C}$, then $F(f):F(\A)\ra F(\p\A)$. Set $G(f)=\theta^{\p\A}\circ F(f)\circ (\theta^{\A})^{-1}$.
	As $\theta^{\A}$, $\theta^{\p\A}$ are uniformly computable from $\A$, respectively $\p\A$ in $\mf{C}$ and $F(f)$ is uniformly computable from $\A\oplus f\oplus \p\A$, there is a Turing operator, say $\Phi_{*}'$, such that
	\[\Phi_{*}^{\prime \A\oplus f\oplus \p\A}=G(f).\]
	It follows that $\Phi_{*}'$ qualifies as the second partial witness of computability of $G$. $G$ is a functor as for $\A\in \mf{C}$, $G(id_{\A})=\theta^{\A}\circ  F(id_\A) \circ (\theta^{\A})^{-1}=\theta^{\A}\circ (\theta^{\A})^{-1}=id_{G(\A)}$ and for $f: \A\ra \p\A, g:\p\A\ra\hat\A \in \mf{C}$,
	\[ G(g\circ f)=\theta^{\hat\A}\circ F(g\circ f)\circ (\theta^{\A})^{-1}
	=\theta^{\hat\A}\circ F(g)\circ (\theta^{\p\A})^{-1} \circ \theta^{\p\A} \circ F(f) \circ (\theta^{\A})^{-1}=G(g)\circ G(f).\]
	As argued above, the function $\theta^{\A}$, which induces the isomorphism between $F(\A)$ and $G(\A)$ is uniformly computable in $\mf{C}$ from $\A$. Hence, there is a Turing functional $\Lambda$ such that $\Lambda^{\A}=\theta^{\A}$. It witnesses the effective isomorphism between $F$ and $G$.
\end{proof}
		A similar result as \cref{th:enumimplcompfunc} holds for enumerable bi-\-transforma\-bility and computable bi-\-transforma\-bility.
		\begin{theorem}\label{th:enumbiimplcompbi}
			Let $\mc{A},\mc{B}$ be enumerably bi-transformable, then they are also computably bi-trans\-form\-able.
		\end{theorem}
		\begin{proof}
			Let  $F: Iso(\A) \ra Iso(\B)$, $G: Iso(\B) \ra Iso(\A)$ be enumerable functors witnessing the enumerable bi-transformability between $\mc{A},\mc{B}$. By \cref{th:enumimplcompfunc} there are computable functors
			\[
				F': Iso(\mc{A}) \ra Iso(\mc{B}) \quad \text{and} \quad G': Iso(\mc{B}) \ra Iso(\mc{A}).
			\]
			Furthermore there are Turing operators $\Theta$ and $\Omega$ inducing the effective isomorphisms between $F$ and $F'$ and $G$ and $G'$ respectively, i.e.,
      $$ \Theta^{\p\A}: F(\p\A) \ra F'(\p\A) \quad\text{and} \quad\Omega^{\p\B}: G(\p\B) \ra G'(\p\B).$$
			Recall the Turing operators $\Lambda_{\A}$ and $\Lambda_{\B}$ witnessing that $F$ and $G$ are pseudo-inverses. For any $\p\A\in Iso(\A)$ and $\p\B \in Iso(\B)$,
      $$\Lambda_{\mc{A}}^{\p\A}: \p\A \ra G(F(\p\A)) \quad\text{and} \quad \Lambda_{\mc{B}}^{\p\B}: \p\B \ra F(G(\p\B)).$$
			Observe that the isomorphisms computed by $\Omega$ given $\p\B\in Iso(\B)$ as oracle are uniformly computable in $Iso(\mc{A})$ because $Iso(\mc{B})$ is uniformly computable in $Iso(\mc{A})$ since $F'$ is a computable functor, and that the analogous statement holds for $\Theta$. Consider the following diagram for any $\p\A \in Iso(\mc{A})$.
			\begin{center}
				\begin{tikzpicture}
					\tikzset{
						>=stealth,
						auto,
						node distance=1cm,
						vertex/.style={circle,fill,align=left, minimum size=5pt, inner sep=0pt},
						every loop/.style={looseness=10}
					}
					\node (a) at (0,0){$\presentation{A}$};
					\node (fa) [below = of a] {$F(\presentation{A})$};
					\node (gfa) [below = of fa] {$G(F(\presentation{A}))$};

					\node (fpa) [right = of a] {$F'(\presentation{A})$};
					\node (gpfpa) [right = of fpa] {$G'(F'(\presentation{A}))$};
					\node (gpfa) [right = of fa] {$G'(F(\presentation{A}))$};

					\path[dashed]
					(a) [->] edge node [left] {$F$} (fa)
					(fa) [->] edge node [left] {$G$} (gfa)
					(a) [->] edge node {$F'$} (fpa)
					(fpa) [->] edge node {$G'$} (gpfpa);
					\path[]
					(fa) [->] edge node [right]{$\Theta^{\p\A}$} (fpa)
					(gfa) [->] edge node [right]{$\Omega^{F(\p\A)}$} (gpfa)
					(gpfa) [->] edge node [right]{$G'(\Theta^\p\A)$} (gpfpa);
					\path[bend right=45 ]
					(a) [->] edge node [left] {$\Lambda_\mc{A}^\p\A$} (gfa);
				\end{tikzpicture}
			\end{center}
			Note that $\Theta^{\p\A}:F(\p\A)\ra F'(\p\A)\in \mf{D}$ and hence $G'(\Theta^{\p\A})$ is an isomorphism from $G'(F(\p\A))$ to $G'(F'(\p\A))$. Analogous diagrams can be drawn for any $\p\B \in Iso(\p\B)$.
			We therefore define $\Gamma_\A^{\p\A}$ and $\Gamma_\B^{\p\B}$ as
					$$\Gamma_\A^{\p\A}=  G'(\Theta^{\p\A}) \circ \Omega^{F(\p\A)}\circ \Lambda_\A^{\p\A} \quad\text{and}\quad
					\Gamma_\B^{\p\B}= F'(\Omega^{\p\B}) \circ \Theta^{G(\p\B)}\circ\Lambda_\B^{\p\B}.$$
			It is easy to see from the above diagram that they induce the wanted isomorphisms $\Gamma_\A^{\p\A}: \A \ra G'(F'(\p\A))$, respectively $\Gamma_\B^{\p\B}:\B\ra F'(G'(\p\B))$ for all $\p\A\in Iso(\A)$ and all $\p\B\in Iso(\B)$.
			Since all functions in their definition are uniformly computable in $\A$, respectively $\B$, $\Gamma_\mc{A}$, $\Gamma_\mc{B}$ witness that $G'\circ F'$ and $F'\circ G'$ are effectively isomorphic to the identity functors $id_\mf{C}$, respectively $id_\mf{D}$. It remains to show that $\Gamma_\mc{B}^{F'(\p\A)}=F'(\Gamma_\A^{\p\A})$ and $\Gamma_\mc{A}^{G'(\p\B)}=G'(\Gamma_\B^{\p\B})$. We will prove the first statement, the proof of the second statement is analogous.

			First recall that by the construction of $F',G'$ in \cref{th:enumimplcompfunc}, for any isomorphism $f:\p\A \ra \hat\A$ between two copies $\p\A, \hat\A$ of $\A$ and for any isomorphism $g:\p\B \ra \hat\B$ between two copies $\p\B, \hat\B$ of $\B$
			\[
				F'(f)=\Theta^{\hat\A}\circ F(f)\circ(\Theta^{\p\A})^{-1}\quad \text{and} \quad G'(g)=\Omega^{\hat\B}\circ G(g)\circ(\Omega^{\p\B})^{-1}
			\]
      because $F,F'$ and $G,G'$ are effectively isomorphic.
			Now, let $\p\A\in Iso(\A)$, then
			\[ G'(\Theta^{\p\A})=\Omega^{F'(\p\A)}\circ G(\Theta^{\p\A})\circ (\Omega^{F(\p\A)})^{-1}.\]
			Therefore $\Gamma_\A^{\p\A}=\Omega^{F'(\p\A)}\circ G(\Theta^{\A})\circ \Lambda^{\p\A}_\A$ and
			\[ F'(\Gamma^{\p\A}_\A)=F'(\Omega^{F'(\p\A)})\circ F'(G(\Theta^{\p\A})) \circ F'(\Lambda_\A^{\p\A}).\]
			Furthermore, $F'(\Lambda_\A^{\p\A})=\Theta^{G(F(\p\A))}\circ F(\Lambda_\A^{\p\A})\circ (\Theta^{\p\A})^{-1}$. It follows that
			\[F'(\Gamma^{\p\A}_\A)=F'(\Omega^{F'(\p\A)})\circ F'(G(\Theta^{\p\A}))\circ\Theta^{G(F(\p\A))}\circ F(\Lambda_\A^{\p\A})\circ (\Theta^{\p\A})^{-1}.\]
			Notice that $F'(G(\Theta^{\p\A}))=\Theta^{G(F'(\p\A))}\circ F(G(\Theta^\p\A))\circ (\Theta^{G(F(\p\A))})^{-1}$, hence
			\[F'(\Gamma^{\p\A}_\A)=F'(\Omega^{F'(\p\A)})\circ \Theta^{G(F'(\p\A))}\circ F(G(\Theta^\p\A))\circ F(\Lambda_\A^{\p\A})\circ (\Theta^{\p\A})^{-1}.\]

			Consider $\Gamma_\B^{F'(\p\A)}=F'(\Omega^{F'(\p\A)}) \circ \Theta^{G(F'(\p\A))}\circ\Lambda_\B^{F'(\p\A)}$ and recall that $\Lambda_\B$ witnesses the effective isomorphism between $id_{Iso(\B)}$ and $F\circ G$. As $F'(\p\A)$ and $F(\p\A)$ are both in $Iso(\B)$ we have by \cref{def:effeciso} that $\Lambda_\B^{F'(\p\A)}\circ \Theta^{\p\A}=F(G(\Theta^{\p\A}))\circ\Lambda_\B^{F(\p\A)}$ and thus
      \[ \Lambda_\B^{F'(\p\A)}=F(G(\Theta^{\p\A}))\circ \Lambda_\B^{F(\p\A)}\circ (\Theta^{\p\A})^{-1}.\]
      Because $\Lambda_\B^{F(\p\A)}=F(\Lambda_\A^{\p\A})$
			\[\Gamma_\B^{F'(\p\A)}=F'(\Omega^{F'(\p\A)}) \circ \Theta^{G(F'(\A))}\circ F(G(\Theta^{\p\A}))\circ F(\Lambda_\A^{\p\A})\circ {(\Theta^{\p\A})}^{-1}=F'(\Gamma_\A^{\p\A}).\]
			By the same argument $G'(\Gamma^{\p\B}_\B)=\Gamma^{G'(\p\B)}_\A$ for all $\p\B\in Iso(\B)$. It follows that $F'$ and $G'$ are pseudo-inverses.
		\end{proof}
    By adapting the proof of \cref{th:enumbiimplcompbi} we get the same result for u.e.t.~reduction.
		\begin{corollary}\label{cor:enumimpliescomp}
			Let $\mf{C}$ be uniformly enumerably transformally reducible to $\mf{D}$, then $\mf{C}$ is uniformly computably transformally reducible to $\mf{D}$.
		\end{corollary}
\section{Enumerable Functors and Effective Interpretability}\label{sec:effectiveint}
        The main goal of this section is to prove the following theorem, exhibiting an equivalence between enumerable functors and a restricted version of effective interpetability.
        \begin{theorem}\label{th:equivalenceintfunc}
					A structure $\A$ is effectively interpretable in $\B$ with $\sim$ computable if and only if there is an enumerable functor $F:Iso(\B) \ra Iso(\A)$.
				\end{theorem}
        We prove \cref{th:equivalenceintfunc} constructively and furthermore show that given a functor $F$, the functor $I^F$ obtained by using the procedures we give in the proof is effectively isomorphic to $F$.
        \begin{proposition}\label{prop:enumeffisoif}
		   		Let $F:Iso(\B) \ra Iso(\A)$ be an enumerable functor. Then $F$ and $I^F$ are effectively isomorphic.
		   	\end{proposition}
        We also prove statements analogous to \cref{th:equivalenceintfunc} for enumerable bi-transformability and effective bi-interpretability, and reducibility by uniform enumerable transformations and reducibility via effective bi-interpretability.

        The authors of~\cite{harrison-trainor_computable_2017} proved similar results for computable functors and effective interpretability.

      	Before we give the proofs we recall some definitions.
				\subsection{Effective Interpretability}
        \begin{definition}
           A relation $R$ is \emph{uniformly intrinsically computable, short u.r.i.\ computable,} in $\A$ if there is Turing operator $\Phi$ such that $\Phi^{\p\A}=R^{\p\A}$ for any $\p\A\in Iso(\A)$.

           \noindent A relation $R$ is \emph{uniformly intrinsically computably enumerable, short u.r.i.c.e.,} in $\A$ if there is a Turing operator $\Phi$ such that $R^{\p\A} = \rng(\Phi^{\p\A})$ for any $\p\A\in Iso(\A)$.
        \end{definition}
        We say that a relation is $\Sicom{1}$-definable in the language $L$ if it is definable by a $\Sigma_1$ computable infinitary formula without parameters in $L$. A relation is $\Decom{1}$-definable if it and its corelation are definable by a $\Sicom{1}$ formula.
        That a relation $R$ is $\Sicom{1}$-definable ($\Decom{1}$-definable) in a structure $\A$ is strongly connected to it being  u.r.i.c.e.\ (u.r.i.\ computable).
        Ash, Knight, and Slaman~\cite{ash_relatively_1993}, building on work by Ash, Knight, Manasse, and Slaman~\cite{ash_generic_1989} and Chisholm~\cite{chisholm_effective_1990}, proved that a relation $R$ is u.r.i.c.e.\ (u.r.i.\ computable) in $\A$ iff it is $\Sicom{1}$-definable ($\Decom{1}$-definable) in $\A$.

				In~\cite{montalban_rice_2012} Montalb\'an studied the algorithmic complexity of sequences of relations. Following his definition, a sequence of relations $(R_i)_{i\in \omega}$ is u.r.i.\  computable in $\A$ if the set $\oplus_{i\in \omega} R_i$ is u.r.i.\ computable in $\A$. By the work of Ash, Knight, and Slaman this is the case iff $\oplus_{i\in \omega} R_i$ is $\Decom{1}$-definable. Thus, a sequence of relations $(R_i)_{i\in \omega}$ is $\Decom{1}$-definable if $\oplus_{i\in \omega} R_i$ is.
				\begin{definition}\label{definition:effectiveinterp}
					A structure $\mc{A}=(A, P_0^\mc{A},P_1^\mc{A},\cdots)$ is \textit{effectively interpretable} in $\mc{B}$ if there exists a $\Decom{1}$-definable sequence of relations (in $\mc{B}$) $(\eidomain{A}{B}, \sim,R_0,R_1,\cdots)$ such that
					\begin{thenum}
						\item $\eidomain{A}{B} \subseteq B^{<\omega}$,
						\item $\sim$ is an equivalence relation on $\eidomain{A}{B}$,
						\item $R_i \subseteq {(B^{<\omega})}^{a_{R_i}}$ is closed under $\sim$ within $\eidomain{A}{B}$,
					\end{thenum}
					and there exists a function $f^\mc{B}_\mc{A}:\eidomain{A}{B} \ra \mc{A}$, the \textit{effective interpretation} of $\mc{A}$ in $\mc{B}$, which induces an isomorphism:
					\[(\eidomain{A}{B},R_0,R_1,\cdots)/_\sim \cong (A,P_0^\mc{A},P_1^\mc{A},\cdots)\]
				\end{definition}
        We use the definition from~\cite{harrison-trainor_computable_2017}, in the literature~\cite{montalban_rice_2012,montalban_computability_2014}, effective interpretability is sometimes defined differently with $\eidomain{A}{B}$ required to be $\Sicom{1}$-definable instead of $\Decom{1}$-definable. These two definitions are equivalent; in our proof of \cref{prop:enumimplint} we demonstrate how to transform an effective interpretation where $\eidomain{A}{B}$ is $\Sicom{1}$-definable into one where it is $\Decom{1}$-definable.

        In \cref{th:enumimplcompfunc} we do not only use effective interpretability but we also require the equivalence relation in the definition to be computable. The following proposition shows that this is justified.
        \begin{proposition}
          Let $(R_i)_{i\in \omega}$ be $\Decom{1}$-definable in $\A$ and let $X$ be a computable set. Then $(X,R_1,R_2,\dots)$ is $\Decom{1}$-definable in $\A$.
        \end{proposition}
        \begin{proof}
          Let $(R_i)_{i\in \omega}$ be $\Decom{1}$-definable in $\A$ and let $X$ be a computable set, say it is computed by $\phi_e$, and let $\Phi$ be the Turing operator witnessing that $(R_i)_{i\in \omega}$ is u.r.i.\ computable. Now define a new operator by
          $$ \Phi'^{S}(\langle i, x \rangle)=
          \begin{cases}
            \Phi^{S}(\langle i-1,x\rangle) & i>1\\
            \phi_e(x) & \text{otherwise}
          \end{cases}
          .$$
          Clearly, $\Phi'$ is a computable operator and witnesses that the sequence $(X, R_1,R_2,\dots)$ is u.r.i.\ computable in $\A$.
        \end{proof}
				Several possibilities to define an equivalence between structures based on effective interpretations exist.
				One is the notion of $\Sigma$-equivalence investigated in~\cite{stukachev_effective_2013}, where two structures are $\Sigma$-equivalent if they are $\Sigma$-definable in each other. We will look at a stronger notion, effective bi-interpretability, which additionally requires the composition of the interpretations to be computable in the respective structures. This was first studied by Montalb\'an~\cite{montalban_computability_2014}.
				\begin{definition}\label{definition:effectivebiint}
					Two structures $\mc{A}$ and $\mc{B}$ are \textit{effectively bi-interpretable} if there are effective interpretations of one in the other such that the compositions
					\[ f_\B^\A\circ \p f^\mc{B}_\mc{A}: \mc{D}om_\mc{B}^{\mc{D}om^\mc{B}_\mc{A}}\ra \mc{B} \quad \mbox{and} \quad f_\mc{A}^\mc{B} \circ \p f_\mc{B}^\mc{A}: \mc{D}om_\mc{A}^{\mc{D}om^\mc{A}_\mc{B}} \ra \mc{A} \]
					are uniformly relatively intrinsically computable in $\mc{B}$ and $\mc{A}$ respectively. (Here the function $\tilde{f}^\mc{B}_\mc{A}:(\eidomain{A}{B})^{<\omega}\ra\mc{A}^{<\omega}$ is the canonic extension of $f_\mc{A}^\mc{B}: \mc{D}om^\mc{B}_\mc{A} \ra \mc{A}$ mapping $\mc{D}om^{\mc{D}om^\mc{B}_\mc{A}}_\mc{B}$ to $\mc{D}om^\mc{A}_\mc{B}$.)
				\end{definition}
				In line with the definition of $\p f^{\B}_\A$ in \cref{definition:effectivebiint}, for a function $f:\A\ra\B$, $\p f$ is the canonic extension of $f$ to tuples, i.e.,
				\[\p f: \A^{<\omega} \ra \B^{<\omega} \text{ with } \p f((x_1,\dots))= (f(x_1),\dots).\]

				As for computable and enumerable functors one gets a method of reduction of classes of structures based on effective bi-interpretability~\cite{montalban_computability_2014}.
				\begin{definition}
					A class \emph{$\mf{C}$ is reducible to $\mf{D}$ via effective bi-interpretability} if there are $\Decom{1}$-formulas such that for every $\mc{A} \in \mf{C}$, there is a $\mc{B} \in \mf{D}$ such that $\mc{A}$ and $\mc{B}$ are effectively bi-interpretable using those formulas and the formulas are independent of the choice of $\A,\B$. A class $\mf{C}$ is \emph{complete for effective bi-interpretability}, if for every computable language $L$, the class of $L$-structures is reducible to $\mf{C}$ via effective bi-interpretability.
				\end{definition}
				\subsection{Proof of \cref{th:equivalenceintfunc} and \cref{prop:enumeffisoif}}
				We prove the two directions of the equivalence in \cref{th:equivalenceintfunc} separately in \cref{prop:intimplenum} and \cref{prop:enumimplint}.
				\begin{proposition}\label{prop:intimplenum}
					If $\A$ is effectively interpretable in $\B$ with $\sim$ computable, then there is an enumerable functor $F:Iso(\B) \ra Iso(\A)$.
				\end{proposition}
				\begin{proof}
					Let $\mc{A}$ be effectively interpretable in $\mc{B}$ and $\sim$ computable using the same notation as in \cref{definition:effectiveinterp}. We will construct $F$ by giving two witnesses $(\Psi,\Phi_*)$ for it.

					By assumption the languages $L_\mc{A}$ and $L_\mc{B}$ are computable. Hence the set $\mc{D}_{L_\mc{B}}\subseteq \omega^{<\omega}$ of all possible finite atomic diagrams in $L_\mc{B}$ is computable.
					For $(\eidomain{A}{B},R_1,\dots)$, let their defining $\Sicom 1$ formulas be $(\phi_{\eidomain{A}{B}}, \phi_{\neg \eidomain{A}{B}},\phi_{R_1},\phi_{\neg R_1},\dots)$; notice that we also use $\phi_{\neg R_i}$, the defining formula of the complement of $R_i$.
					Fix a computable bijection $\sigma: \omega \ra \omega^{<\omega}$ and define the function $h:\omega^{<\omega} \ra \omega$ by
					\begin{equation} \tag{1}\label{eq:defh} h(\ol{y}) = \mu x [ \sigma(x) \sim \ol{y}] = \mu x\leq \sigma^{-1}(\ol y) [\sigma(x) \sim \ol y ] .\end{equation}
					Intuitively $h$ maps any tuple $\ol{y}$ to a fixed presentation of it under $\sim$. We use the minimal presentation in the order induced by $\sigma$ to make $h$ computable.
					We now build $\Psi$ using $h$ in the following way.
					\begin{equation}\tag{2}\label{eq:psidef1}
						(\alpha,  x =x ) \in \Psi \LR \exists \ol{y}\ x=h(\ol y) \land \alpha \in \mc{D}_{L_\mc{B}} \land \alpha \models \phi_{\eidomain{A}{B}}(\ol y)
					\end{equation}
					Let $p_i$ be the arity of the relation $P_i$, then
					\begin{align}\tag{3}\label{eq:psidef2}
						(\alpha,  P_i(x_1,\dots,x_{p_i}) )\in \Psi      & \LR \exists \ol{y}_1,\dots, \ol{y}_{p_i} \left( \bigwedge_{i\in \{1\dots p_i\}} x_i=h(\ol y_i)\right) \land \alpha \in \mc{D}_{L_\mc{B}}\land \alpha \models \phi_{R_i}(\ol y_1 ,\dots, \ol{y}_{p_i})       \\ 
						\tag{4}\label{eq:psidef3}
						(\alpha,  \neg P_i(x_1,\dots,x_{p_i}) )\in \Psi & \LR \exists \ol{y}_1,\dots, \ol{y}_{p_i}\left( \bigwedge_{i\in \{1\dots p_i\}} x_i=h(\ol y_i)\right) \land \alpha \in \mc{D}_{L_\mc{B}} \land \alpha \models \phi_{\neg R_i}(\ol y_1 ,\dots, \ol{y}_{p_i})
					\end{align}
					Notice that the problem of deciding whether a finite structure in a computable language is a model of a $\Sicom{1}$-formula is c.e. It follows that $\Psi$ defined by \cref{eq:psidef1,eq:psidef2,eq:psidef3} is c.e.

					We now show that for $\presentation{B} \in Iso(\mc{B})$, its image $\Psi^{\p\B}$ is in the isomorphism class of $\A$.
					Define $(\eidomain{A}{\p{B}},\sim, R_1^{\p\B},\dots)$ in the obvious way using the formulas of the effective interpretation of $\A$ in $\B$.
          Say $g$ is an isomorphism from $\B$ to $\p\B$,
					then $(\eidomain{A}{B}, R_1^\B,\dots)/_\sim\cong_{\p g}(\eidomain{A}{\p B}, R_1^{\p\B},\dots)/_\sim$.
					Recall the function $h$ used in the definition of $\Psi$. Let the function $\xi_{\p\B}: (\eidomain{A}{\p B}, R_1^{\p\B},\dots)/_\sim \ra \Psi^{\p\B }$ be the canonical restriction of $h$ to the quotient of the domain, i.e., $\xi_{\p\B}({[\ol{y}]}_\sim)=h(\ol{y})$. It follows from \cref{eq:defh} that $\xi_{\p\B}$ is well defined. We will show that $\xi_{\p\B}$ is an isomorphism.

					\begin{itemize}
						\item $\xi_{\p\B}$ is $1-1$ because by \cref{eq:defh} if $h(\ol y)=h(\ol x)$ then $\ol x \sim \ol y$. Hence, $[\ol y]_\sim =[\ol x]_\sim$.
            \item $\xi_{\p\B}$ is onto because by \cref{eq:psidef1} if $x\in \Psi^{\p\B }$, then $\exists \ol{y} \in \eidomain{A}{\p\B}$ such that  $x=h(\ol{y})$.
					\end{itemize}
					It follows that $\xi_{\p\B}$ is bijective. By \cref{eq:psidef2,eq:psidef3} $\xi_{\p\B}$ is an homomorphism and therefore by the above arguments also an isomorphism. Hence, $\Psi^{\p\B }\in Iso(\A)$ as $\xi_{\p\B}\circ  \p g \circ (f_{\A}^{\B})^{-1}$ is an isomorphism from $\A$ to $\Psi^{\p\B}$.
          Notice that $\xi_{\p\B}$ is computable from $\p\B $ and that the computation is uniform.

					We now build $\Phi_*$. Assume $\p\B\cong_f \hat\B$; we use the extension of $f:\omega \ra \omega$, $\p f :\omega^{<\omega}\ra\omega^{<\omega}$ and set $F(f)=\xi_{\hat\B}\circ \p f\circ \xi_{\p\B}^{-1}$. Because $\xi_{\hat\B}$, and $\xi_{\p\B}^{-1}$ are uniformly computable in $Iso(\B)$ and $\p f$ is uniformly computable from $f$, there is a Turing operator $\Phi_*$ such that $\Phi_*^{\p\B \oplus f\oplus \hat\B}=F(f)$. Furthermore, $F(f)$ is a bijection because so are the functions it is composed of. Moreover, $F(\presentation{B})\cong_{\xi_{\hat\B }\circ \p f\circ  \xi_{\p\B}^{-1}} F(\hat\B)$ because for $Q\in (\eidomain{A}{\mc{B}},R_1,\neg R_1,\dots )$, $\presentation{B} \models \phi_{Q}(\overline{x})$ if and only if $\hat\B \models \phi_{Q}(\p f(\overline{x}))$.
				\end{proof}

				\begin{proposition}\label{prop:enumimplint}

					If there is an enumerable functor $F:Iso(\mc{B}) \ra Iso(\mc{A})$, then $\mc{A}$ is effectively interpretable in $\mc{B}$ with the restriction that $\sim$ is computable.
				\end{proposition}
				\begin{proof}
					Assume $F$ is witnessed by $(\Psi,\Phi_*)$. We will first provide definitions of $\eidomain{A}{B}$ and relations $R_i$ which are $\Sicom{1}$-definable in $\B$. We then use these to build an interpretation having the desired properties, i.e., the sequence of relations is $\Decom{1}$-definable and the equivalence relation is computable. We apply the standard argument that the definition of effective interpretability which requires $\eidomain{A}{B}$ to be $\Sicom{1}$-definable is equivalent to the definition we use.

					In what follows we write $\B\restrict[A]$ for the substructure of $\B$ induced by the restriction of its universe to elements in the set $A$. Similarly $\B\restrict[\ol{a}]$ is the substructure of $\B$ induced by the restriction of its universe to elements in the tuple $\ol{a}$.

					The $\Sicom{1}$-definable interpretation made of $\eidomain{A}{B}, \sim$ and the sequence of relations $(R_1,\dots )$ are defined as follows.
					\begin{enumerate}[leftmargin=1.2cm]
						\item[$\eidomain{A}{B}$:] The domain is a subset of $\omega^{<\omega}\times \omega$ such that
						      \[ (\ol a, i) \in \eidomain{A}{B} \LR  i=i\in \Psi^{\B\upharpoonright_{\ol{a}}}.\]
						      Since $\Psi$ is c.e.~and the restriction of $\B$ to $\ol{a}$ is computable relative to $\B$, $\eidomain{A}{B}$ is uniformly r.i.c.e.~and therefore also $\Sicom{1}$-definable in $\B$.

						\item[$\sim$:] For all $(\ol a, i), (\ol b, j)\in \omega^{<\omega}\times \omega$,
						      \[ (\ol a, i) \sim (\ol b, j) \LR i=j.\]
						      By definition $\sim$ is computable, reflexive, symmetric, and transitive.
						\item[$R_i$:] Let $P_i$ have arity $p_i$. Then for all $(\ol{a}_1,x_1),\dots,(\ol{a}_{p_i},x_{p_i})\in \eidomain{A}{B}$ we define $R_i$ as follows.
						      \[ ((\ol a_1,x_1), \dots ,(\ol a_{p_i}, x_{p_i})) \in R_i  \LR P_i(x_1,\dots, x_{p_i}) \in \Psi^{\B}\]
						      \[ ((\ol a_1,x_1), \dots ,(\ol a_{p_i}, x_{p_i})) \not \in R_i  \LR \neg P_i(x_1,\dots, x_{p_i})  \in \Psi^{\B}\]
						      For $(\ol{a}_1,x_1),\dots,(\ol{a}_{p_i},x_{p_i})\in \eidomain{A}{B}$ either $P_i(x_1,\dots, x_n)$ or $ \neg P_i(x_1,\dots,x_n)$ is in $\Psi^{\B}$ and $\eidomain{A}{B}$ is $\Sicom{1}$-definable. Therefore $R_i$ is also $\Sicom{1}$-definable uniformly in $i$.
					\end{enumerate}
					Because $\sim$ is computable the restriction to the domain is trivially $\Sicom{1}$-definable. Hence, also the sequence $(\eidomain{A}{B},\sim, R_1,\dots )$ is $\Sicom{1}$-definable.
					\begin{claim}\label{claim:compatible}
						The equivalence relation $\sim$ is compatible with the definition of $R_i$, i.e., if for all $(\ol{a}_1,k_1), \ldots, (\ol{a}_{p_i},k_{p_i}), (\ol{b}_1,l_1),\ldots, (\ol{b}_{p_i},l_{p_i}) \in \eidomain{A}{B}$,  $(\ol{a}_1,k_1)\sim(\ol{b}_1,l_1),\ldots,(\ol{a}_{p_i},k_{p_i})\sim(\ol{b}_{p_i},l_{p_i})$, then $((\ol{a}_1,k_1), \ldots, (\ol{a}_{p_i},k_{p_i})) \in R_i$ iff $((\ol{b}_1,l_1),\ldots (\ol{b}_{p_i},l_{p_i}))\in R_i$.
					\end{claim}
					\begin{proof} The claim follows from the definitions of $R_i$ and $\sim$ because for $i\in \{1,\dots,p_i\}$, $(\ol{a}_i,k_i) \sim (\ol{b}_i,l_i)$ if and only $k_i$ is equal to $l_i$.				\end{proof}
					Consider the function
					$f:(\eidomain{A}{B}, R_1,\dots)/_\sim\ra F(\B)$ defined as
					$f([(\ol{a},x)]_{\sim})= x$.
					We claim that $(\eidomain{A}{B},R_1,\dots)/_\sim\cong_f F(\B)$. The function $f$ is a bijection by the definition of $\eidomain{A}{B}$ and $\sim$. It follows from the definition of $R_i$ and \cref{claim:compatible} that $f$ is an isomorphism. We defined everything needed for an effective interpretation with the exception that $(\eidomain{A}{B},\sim,R_1,\dots)$ is not $\Decom{1}$-definable.

					We now define a sequence of relations $(\eidomstar{A}{B}, \sim^*, R_1^*,\dots)$ $\Decom{1}$-definable in $\B$ such that the structure $(\eidomstar{A}{B}, R_1^*, \dots)/_{\sim^*} $ is isomorphic to $(\eidomain{A}{B}, R_1,\dots)/_\sim$.
					This sequence of relations is an effective interpretation of $\A$ in $\B$.
					\begin{enumerate}[leftmargin=1.3cm]
						\item[$\eidomstar{A}{B}$:] Since the original domain $\eidomain{A}{B}$ is $\Sicom{1}$-definable, every element $(\ol{a},i)$ satisfies a finitary existential formula $\exists \ol{y} \phi_j(\ol{a},i,\ol{y})$ in the infinite disjunction defining $\eidomain{A}{B}$ where $j$ is the index of the formula in some computable enumeration. The new domain $\eidomstar{A}{B}$ is a subset of $\omega^{<\omega}\times \omega \times \omega^{<\omega}\times \omega$ defined as follows.
						      \[ (\ol{a},i,\ol{y},j)\in \eidomstar{A}{B} \LR \B \models \phi_j(\ol{a},i,\ol{y})\]
						      $\eidomstar{A}{B}$ is clearly uniformly r.i.\ computable and thus $\Decom{1}$-definable in the language of $\B$.
						\item[$\sim^*$:]
						      For all $(\ol{a},i,\ol{y},j), (\ol{b},k,\ol{z},l)\in \omega^{<\omega}\times \omega \times \omega^{<\omega}\times \omega$
						      \[ (\ol{a},i,\ol{y},j)\sim^* (\ol{b},k,\ol{z},l)\LR \ol{a}=\ol{b} \land i=k.\]
						      This is by definition a computable equivalence relation.
						\item[$R_i^*$:] As above let $P_i$ have arity $p_i$. Then for all $(\ol{a}_1,k_1,\ol{y}_1,j_1),\dots, (\ol{a}_{p_i},k_{p_i},\ol{y}_{p_i},j_{p_i})\in \eidomstar{A}{B}$, $R_i^*$ is defined as follows.
						      \[ ((\ol{a}_1,k_1,\ol{y}_1,j_1),\dots, (\ol{a}_{p_i},k_{p_i},\ol{y}_{p_i},j_{p_i})) \in R_i^* \LR
						      	P_i(k_1,\dots, k_{p_i})  \in \Psi^{D(\B)}\]
						      	\[ ((\ol{a}_1,k_1,\ol{y}_1,j_1),\dots, (\ol{a}_{p_i},k_{p_i},\ol{y}_{p_i},j_{p_i})) \not \in R_i^* \LR \neg P_i(k_1,\dots, k_{p_i}) \in \Psi^{D(\B)}\]
						      	By the same arguments as for $R_i$, $R_i^*$ is uniformly relatively intrinsically computable from $\B$ and therefore $\Decom{1}$-definable in the language of $\B$.
						      	\end{enumerate}
						      	The sequence of relations $(\eidomstar{A}{B},\sim^*, R_1^*,\dots )$ is $\Decom{1}$-definable by an argument similar to the argument that the sequence $( \eidomain{A}{B},\sim, R_1,\dots )$ is $\Sicom{1}$-definable.
						      	\begin{claim}
						      		The equivalence relation $\sim^*$ is compatible with the definition of $R_i$.
						      	\end{claim}
						      	\begin{proof}
						      		The claim follows from an argument analogous to that given in \cref{claim:compatible}.
						      	\end{proof}
						      	Define $f^*: (\eidomstar{A}{B},R_1^*,\dots)/_{\sim^*}\ra (\eidomain{A}{B},R_1,\dots)/_\sim$ as $f^*([(\ol{a},x,\ol{y},j)]_{\sim^*})=[(\ol{a}, x)]_\sim$. It is not hard to see that $f^*$ is well defined on the quotient structure $(\eidomstar{A}{B},R_1^*,\dots)/_{\sim^*}$ and induces an isomorphism between it and $(\eidomain{A}{B},R_1,\dots)/_{\sim}$. Therefore $(\eidomstar{A}{B},R_1^*,\dots)/_{\sim^*}$ is isomorphic to $F(\B)$ by $f\circ f^*$. Since $(\eidomstar{A}{B},R_1^*,\dots)$ is $\Decom{1}$-definable and $\sim^*$ is computable, the theorem follows.
						      	\end{proof}

						      	\begin{proof}[Proof of \cref{prop:enumeffisoif}.]
						      		Let $F:\B \ra \A$ be an enumerable functor, $(\eidomain{A}{B},R_1,\dots)/_\sim$ be the effective interpretation one gets by applying the procedure described in \cref{prop:enumimplint} to $F$, and let $\zeta_\B$ be the effective interpretation, i.e., using the definition given in the proof, $\zeta_\B=f\circ f^*$. The function $\zeta_\B$ is computable relative to $\B$ using projection.

						      		We now transform the interpretation back to an enumerable functor using the procedure described in the proof of \cref{prop:intimplenum}. We get a functor $I^F: \B \ra \A$, such that $I^F(\B)\cong_{\xi_\B} (\eidomain{A}{B},R_1,\dots)/_\sim$ by the $\xi_\B$ defined in the proof of \cref{prop:intimplenum}. The following diagram shows the relation between two presentations $\p{\B},\hat{\B}\in Iso(\B)$ isomorphic by $h$ under the two functors.
						      		\begin{center}
						      			\begin{tikzpicture}
						      				\tikzset{
						      					>=stealth,
						      					auto,
						      					node distance=1.5cm,
						      					vertex/.style={circle,fill,align=center, minimum size=5pt, inner sep=0pt},
						      					every loop/.style={looseness=10}
						      				}
						      				\node (fb) at (0,0){$F(\p\B)$};
						      				\node (pb) [left =of fb]{$\p\B$};
						      				\node (p2b) [below= of pb]{$\hat\B$ };
						      				\node (fb2) [below = of fb] {$F(\hat\B)$};

						      				\node (domb2) [below = 0.75cm of p2b]{$(\eidomain{A}{\hat{B}},R_1^{\hat\B},\dots)/_\sim$};
						      				\node (domb) [above =0.75cm of pb] {$(\eidomain{A}{\p B},R_1^{\p\B},\dots)/_\sim$};
						      				\node (ib) [left= of pb] {$I^F(\p\B)$};
						      				\node (ib2) [left=of p2b] {$I^F(\hat\B)$};

						      				\path[dashed]
						      				(pb)[->] edge node {$F$} (fb)
						      				(p2b)[->] edge node{$F$} (fb2)
						      				(pb)[->] edge node [above] {$I^F$}(ib)
						      				(p2b)[->] edge node[above] {$I^F$}(ib2)
						      				;
						      				\path[]
						      				(pb)[->] edge node [left] {$h$} (p2b)
						      				(fb) [->] edge node [left] {$F(h)$} (fb2)
						      				(fb2)[<-] edge node {$\zeta_{\hat\B}$} (domb2)
						      				(ib)[->] edge node {$I^F(h)$} (ib2)
						      				(ib2) [<-] edge node [left,below=0.03cm] {$\xi_{\hat\B}$} (domb2)
						      				;
						      				\path[]
						      				(fb)[<-] edge node [right=0.1cm,above=0.02cm]{$\zeta_{\p\B}$} (domb);
						      				\path[]
						      				(ib) [<-] edge node [above]{$\xi_\p\B$} (domb);
						      			\end{tikzpicture}
						      		\end{center}
						      		By the above diagram for every presentation $\p\B$ of $\B$, $I^F(\p\B)$ and $F(\p\B)$ are isomorphic by $\xi^{-1}_{\p\B}\circ \zeta_{\p\B}$. Also the squares as given in \cref{def:effeciso} for any two presentations $\p\B,\hat\B$ can be seen to commute by the above diagram. Hence, $I^F$ and $F$ are naturally isomorphic. Since the functions $\xi_\B$ and $\zeta_\B$ are both uniformly computable in $Iso(\B)$, $I^F$ and $F$ are effectively isomorphic.
						      	\end{proof}
						      	\subsection{Enumerable bi-transformability and u.e.t.~reductions}
						      	\begin{theorem}\label{th:enumbiiffeffbi}
						      		$\mc{A}$ and $\mc{B}$ are enumerably bi-transformable iff they are effectively bi-interpretable with the restriction that the equivalence relations $\sim$ of the interpretations are computable.
						      	\end{theorem}
						      	Using the proof of \cref{prop:intimplenum} and \cref{prop:enumimplint} the proof of this theorem is similar to the proof of the statement that computable bi-transformability and effective bi-interpretability are equivalent~\cite[Theorem 1.9]{harrison-trainor_computable_2017}. We therefore omit it here and refer the reader to~\cite[Section 4]{harrison-trainor_computable_2017} for a detailed proof.

						      	Since the proof given in~\cite{harrison-trainor_computable_2017} is uniform we get that the same holds for reduction by uniform enumerable transformation and reduction by effective bi-interpretability if we assume that the structures in the classes have universe $\omega$.
						      	\begin{corollary}\label{cor:ueteffbi}
						      		$\mf{C}$ is uniformly enumerably transformally reducible to $\mf{D}$ iff $\mf{C}$ is reducible by effective bi-interpretability to $\mf{D}$ with the restriction that the equivalence relations $\sim$ of the interpretations are computable.
						      	\end{corollary}
                    \section{Conclusion and open questions}
                    It follows from our results in \cref{sec:enumfunc} that enumerable functors and u.e.t.\ reduction preserve all properties preserved by computable functors and u.c.t.\ reduction. Reduction by uniform computable transformations is already a very strong notion of reduction. Montalb\'an~\cite{montalban_computability_2014} showed that bi-interpretability preserves many computability theoretic properties of structures. Using this result and the equivalence to computable functors proven in~\cite{harrison-trainor_computable_2017} we get that computable functors preserve the same properties.

                    It is still open whether the converse of the implications proved in \cref{sec:enumfunc} hold, i.e., if the existence of a computable functor between two structures also implies the existence of an enumerable functor. In light of this we propose the following questions.
                    \begin{question}
                      Are there structures $\A$ and $\B$ such that there is a computable functor $F:Iso(\A)\ra Iso(\B)$ but no enumerable functor?
                    \end{question}
                    \begin{question}
                      Are there computably bi-transformable structures $\A$ and $\B$ which are not enumerably bi-transformable?
                    \end{question}
                    \begin{question}\label{question2}
                      Are there classes of structures $\mf{C}$ and $\mf{D}$ such that $\mf{C}$ is u.c.t.\ reducible to $\mf{D}$ but not u.e.t.\ reducible?
                    \end{question}
                  	\printbibliography%
\end{document}